\documentclass[]{amsart}
\usepackage{amsmath,amsfonts,amssymb}
\usepackage{verbatim}
\usepackage{enumerate}
\usepackage{url}
\usepackage[colorlinks]{hyperref}
\usepackage[latin1]{inputenc}
\usepackage[english]{babel}
\usepackage{tikz}
\usepackage[margin=2.5cm]{geometry}
\def\N{\mathbb N}

\def\Z{\mathbb Z}

\def \F{\mathbb F}
\def \U{\mathcal U}

\def\ord{\mathop{\rm ord}\nolimits}
\def\rad{\mathop{\rm rad}}
\def\lcm{\mathop{\rm lcm}}

\theoremstyle{plain}
\newtheorem{theorem}{Theorem}[section]
\newtheorem{lemma}[theorem]{Lemma}
\newtheorem{definition}[theorem]{Definition}
\newtheorem{corollary}[theorem]{Corollary}

\newtheorem{remark}[theorem]{Remark}
\newtheorem{example}[theorem]{Example}

\def\qed{\hfill\hbox{$\square$}}

\theoremstyle{definition}

\author[F.E. Brochero Mart\'{\i}nez]{F. E. Brochero Mart\'{\i}nez}
\author[Lucas Reis]{Lucas Reis}
\address{
Departamento de Matem\'{a}tica\\
Universidade Federal de Minas Gerais\\
UFMG\\
Belo Horizonte, MG\\
 30123-970\\
 Brazil\\
 }
 \email{fbrocher@mat.ufmg.br }\email{lucasreismat@gmail.com}

\title{ Factoring polynomials of the form $f(x^n)\in \F_q[x]$}

\date{\today
}

\subjclass[2000]{ }


\begin{document}
\maketitle
\begin{abstract}
Let $f(x)\in  \F_q[x]$
be  an irreducible polynomial  of degree $m$ and exponent $e$, and  $n$ be a positive integer  such that $\nu_p(q-1)\ge \nu_{p}(e)+\nu_p(n)$ for all $p$ prime divisor of $n$.   We show a fast algorithm to determine the irreducible factors of $f(x^n)$. We also show the irreducible factors in the case when $\rad(n)$ divides $q-1$ and $gcd(m, n)=1$.   Finally, using this algorithm we split $x^n-1$ into irreducible factors, in the case when $n=2^mp^t$ and $q$ is a generator of the group $\Z_{p^2}^*$. 

\end{abstract}
\section{Introduction}

An important problem in finite field theory is how determine when a polynomial is irreducible, and in the case when the polynomial is reducible, how to find  explicitly  the irreducible factors of the polynomial.
 These problems have  important practical and theoretical consequences in a wide variety of technological situations, including error-correcting codes, cryptography,  efficient and secure communications, deterministic simulations of random processes and digital tracking systems. 

 The following classical result shows necessary and sufficient conditions in order to verify the irreducibility of  the polynomial of the form $f(x^n)$.

\begin{theorem}[\cite{LiNi} Theorem 3.35] \label{irreducible} Let $n$ be a positive integer and $f(x)\in   \F_q[x]$ be an irreducible polynomial of degree $m$ and exponent $e$.  Then the polynomial $f(x^n)$ is irreducible over $  \F_q$ if and only if the following conditions are satisfied:
\begin{enumerate}
\item $rad(n)$ divides $e$;
\item gcd$(n, (q^m-1)/e)=1$ and
\item  if $4|n$ then $4|q^m-1$.
\end{enumerate}
In addition, in the case when the polynomial $f(x^n)$ is irreducible, then it has degree $mn$ and exponent $en$.
\end{theorem}

Observe, in particular, this theorem gives necessary and sufficient condition to determine that  polynomial $x^n-a\in \F_q[x]$ is irreducible (see Theorem 3.75 in \cite{LiNi}).  In addition, if $a=1$ and $n>1$, the polynomial $x^n-1$ is always reducible, and  the problem of finding the irreducible factors of  $x^n-1\in \F_q[x]$ are  strongly related with the problem of finding  irreducible factor of cyclotomic polynomial $\Phi_n(x)$. In fact  $x^n-1=\prod_{d|n} \Phi_d(x)$, where $\Phi_d(x)$ denotes  the $d$-th cyclotomic polynomial (see \cite{LiNi} Theorem 2.45).
In general,  a ``generic efficient algorithm''  to split   $\Phi_n(x)$ in $\F_q[x]$ for arbitrary $n$ and $q$ is an open problem and just some particular cases are known. We note that the problem of determining the irreducible factors of $x^ n-1$ and $\Phi_n(x)$ in $\F_q[x]$  has been consider for several author: Mey \cite{Mey} and Blake, Gao, Mullin \cite{BGM} consider the case when $n=2^m$; 
 Fitzgerald  and Yucas \cite{FiYu} found  explicit factors of ciclotomic polynomial $\Phi_{2^ m3}(x)$ and studied the polynomial $\Phi_{2^ n r}(x)$ in the case when $r$ is prime and $r|(q-1)$; 
the same problem was studied by  Wang and Wang for the  cyclotomic polynomial $\Phi_{2^m5}(x)$ and Tuxanidy and  Wang \cite{TuWa} find the irreducible factors when the factorization of $\Phi_r(x) $ is known;
 Chen, Li and Tuerhong \cite{CLT}  consider the polynomial $x^{2^mp^t}-1$ in the case that  $p|(q-1)$ and in \cite{BGO} the authors found the explicit factorization of $x^n-1$, in the case when $\rad(n)|(q-1)$.

In this article, we are going to consider $f(x)$ and $n$ as in Theorem   \ref{irreducible}, but   satisfied some special condition  in order that  condition (1) or (2) of Theorem is not true. Therefore $f(x^n)$ is reducible polynomial and we show a computational fast algorithm to find every irreducible factors. 

Finally, as application,  we use our method to split $x^{2^mp^t}-1$ into irreducible factors  in the case when $p$ is a prime and $q$ is a primitive element of the multiplicative group $\Z_{p^l}^*$, where $l=\min\{2, t\}$. In fact, this is a particular case of the results obtaion by 
 Tuxanidy and  Wang, but with the advantage that each factor is explicitly shown.

\section{Preliminaries}
Throughout this paper, $\F_q$ will denote the finite field with $q$ elements
. For all prime $p$ and $k\ne 0$ integer, $\nu_p(k)$ denotes the maximum power of $p$ that divides $k$. In the case when $gcd(k,n)=1$, $\ord_k(n)$ denotes the order of $n$  modulo $k$ and for all $\beta\in \F_q^*$, $\ord(\beta)$ denotes the order of $\beta$ in the multiplicative group $\F_q^ *$. Finally,   $\U(k):=\{\beta\in\F_q; \ord(\beta)=k\}$, i.e., $\U(k)$ is the set of primitive $k-$th roots of unity in $\F_q$. 
Notice that $\U(k)\ne \emptyset$ if and only if $k|(q-1)$ and, in this case, $|\U(k)|=\varphi(k)$.

For each monic polynomial $f(x)\in \F_q[x]$ with $f(0)\ne 0$, the exponent $e$ of $f(x)$ is the minimal positive integer such that  $f(x)$ divides $x^e-1$. Since $f(x)$ divides $x^{q^m-1}-1$, where $m$ is the degree of $f$, it  follows that $e$ divides $q^m-1$.
 The following lemma is a classical result about the structure of the roots of an irreducible polynomial in $\F_q$.
\begin{lemma}\label{frobenius} Let $f(x)\in  \F_q[x]$ be a monic  irreducible polynomial of degree $m$ and exponent $e$,  and $\alpha$ be any root of $f(x)$. Then $f(x)=\prod\limits_{j=1}^{m}(x-\alpha^{q^j})$ and $m=\ord_e q$.
\end{lemma}
\begin{proof}
Let $\alpha$ be an arbitrary root of $f(x)=0$ and $u=\ord_e q$.  Observe that  $\alpha\in \F_{q^m}$ is a primitive $e-$th root of unity. Let define the polynomial  $$g(x):=\prod_{j=1}^{u}(x-\alpha^{q^j}).$$
Since the coefficients  of $g(x)$ are invariant by the Frobenius automorphism 
$$\begin{matrix}\tau:&\F_{q^m}&\to& \F_{q^m}\\ &\beta&\mapsto&\beta^q,\end{matrix}$$
it follows that  $g(x)\in  \F_q[x]$. In addition, since $\tau^{j}(f(x))=f(x)$ for all $j\in \N$, we conclude that $\alpha^{q^j}$ are also roots of $f(x)$ for all $j$. 
So,  $g(x)$ divides $f(x)$. But, by hypothesis,  $f(x)$ is irreducible,  therefore $f(x)=g(x)$ and $m=\deg f(x)=\deg g(x)=\ord_e q$ as we want to prove.\qed
\end{proof}

The following Lemma shows a case when is possible to  find every irreducible  factor of $f(x^{p^t})$,  when one irreducible factor is known.

\begin{lemma}\label{fatores} Let $f(x)$ be an irreducible polynomial of degree $m$ and exponent $e$. Let  $n>1$ be  a positive  divisor  of $q-1$ and $\nu_p(n)+\nu_p(e)\le \nu_p(q-1)+\nu_p(\ord_{r_p} q)$,
for all $p$ prime divisor of $n$,  where $r_p$ is the largest divisor of $e$ prime with $p$.  
Then the polynomial $f(x^{n})$ splits as  product of  $n$ irreducible polynomials of  degree $m$.
 In addition, if $g(x)$ is any monic irreducible factor of $f(x^{n})$ and $c$ is any element of  $\U(n)$, then
$$f(x^{n})=\prod_{i=0}^{n-1}[c^{-mj}g(c^jx)]$$
is the factorization of $f(x^{n})$ into irreducible factors.
\end{lemma}

\begin{proof}
It is enoght to prove the Lemma in the case when $n=p^t$. 
Let $g(x)\in \F_q[x]$ be any irreducible factor of $f(x^{p^t})$. 
 Since every root of $f(x)$ is in $\U(e)$, then every root of $g(x)$ is in $\U(p^te)$  and therefore the exponent of $g(x)$ is $p^te$.
 In addition,  by the hypotheses,  $p^t$ divides $q-1$ and $\nu_p(q-1)+\nu_p(\ord_r q)\ge k+l$. Therefore $$\nu_p(\ord_{p^{k+t}} q)=\max\{0, k+t-\nu_p(q-1)\}\le \nu_p(\ord_r q)$$
and
$$\ord_{p^te} q=\ord_{p^{t+k}r} q=\lcm( \ord_{p^{t+k}} q ,\ord_{r} q)=m.$$ 
By Lemma \ref{frobenius}, we conclude that $\deg f(x)=\deg g(x)=m$. Thus, any irreducible factor of $f(x^{p^t})$ has the same degree $m=\deg f(x)$ and since $f(x^{p^t})$ has degree $mp^t$, there are exactly $p^t$ factors.

On the other hand, if $\theta$ is a primitive $(q-1)$-th root of unity, then $\theta^{(q-1)/p^t}\in \U(p^t)$, consequently $\U(p^t)\ne \emptyset$. 
Let $c$ be any element of $\U(p^t)$ and $g(x)$ be any monic irreducible factor of $f(x^{p^t})$. 
So $g(x)$ divides $f(x^{p^t})$, and, changing  $x$ by $c^jx$ we have $g(c^jx)$ divides $f((c^jx)^{p^t})=f(x^{p^t})$. 
Thereby  $g(c^jx)$ is a polynomial in $\F_q[x]$ that also divides $f(x^{p^t})$.  
Therefore  $[c^{-mj}g(c^jx)]$ is a monic irreducible  factor of  $f(x^{p^t})$.
\par 
Finally, we have to prove that every irreducible factor of $f(x^{p^ t})$ can be generated by this method. 
 In fact, we need to prove  that if $0\le i<j<p^t$ then $[c^{-mj}g(c^jx)]$ and $[c^{-mi}g(c^ix)]$ have no  roots in common. 
Since the polynomials  are  monic and irreducible, we just have to show that they are  different. But, if there exist $0\le j<i<p^t$ such that $[c^{-mj}g(c^jx)]=[c^{-mi}g(c^ix)]$, writing $y=c^{j}x$ we have 
\begin{equation}
\quad c^{-m(i-j)}g(c^{(i-j)}y)=g(y),\label{eq1}
\end{equation}
then, we can suppose without loss of generality that $j=0$. 

Let $\beta$ be any root of $g(x)$. Since the exponent of $g(x)$ is $p^te$, then $\beta$ is a primitive $p^te-$th root of unity and by lemma \ref{frobenius}, we know that 
$$g(x)=\prod_{j=1}^{m}(x-\beta^{q^j}).$$ 
By the equation (\ref{eq1}),  we have that
$c^{-mi}g(c^i\beta)=g(\beta)=0$, i.e. $c^i\beta$ is also a root of $g(x)$ and then $c^i\beta=\beta^{q^l}$ for some integer $1\le l<m$. 
Therefore $\beta^{p^t(q^l-1)}=c^{p^ti}=1$ and then $p^t(q^l-1)$ is a multiple of $p^te$, which is equivalent to $e|q^l-1$.   
Again by Lemma \ref{frobenius}, we know that   $m=ord_eq$, therefore $m$ divides $l$,  but this is a contradiction because  $0<l<m$.
 We conclude that the polynomials  $[c^{-mi}g(c^ix)]$, where $0\le i<p^t$,  are  monic irreducible polynomials with no factor in common, such that each one divides $f(x^{p^t})$. In particular 
$$F(x)=\prod_{i=0}^{p^t-1}[c^{-mj}g(c^jx)]$$
is a monic polynomial that divides $f(x^{p^t})$ and it has degree $mp^t$. Thus $F(x)=f(x^{p^t})$, as desired.\qed
\end{proof}

\begin{remark} 
Observe that, the conditions (1) and (2)  of Theorem \ref{irreducible} can be rewritten as
$$\nu_p(e)\ge 1\quad\text{and}\quad \nu_p(q^m-1)=\nu_p(e)$$
for all $p$ prime divisor of $n$.  Since $\ord_{r_p}q$ divides $\ord_e q=m$ and $$\nu_p(q^m-1)=\begin{cases} \nu_p(q-1)+\nu_p(m) &\text{ if $p$ is odd}\\
\nu_2(q-1) &\text{ if $p=2$ is odd and $m$ is odd}\\
\nu_2(q^2-1)+\nu_2(m)-1  &\text{ if $p=2$ is odd and $m$ is even,}\\
\end{cases} $$ in particular,
$\nu_p(q^m-1)\ge \nu_p(e)+\nu_p(n)\ge \nu_p(e)+1$ for all  $p$ prime divisor of $n$, 
  then  the  condition of Lemma \ref{fatores} is a sufficient (but not necessary) condition in order to  $f(x^n) $  be a reducible polynomial.  
\end{remark}

It is clear that this lemma does not give any information in order to  construct an irreducible factor of $f(x^n)$.  Assuming the  following condition,  that implied the condition of Lemma \ref{fatores}, we are going to show  an algorithm to construct one of those factors.

\begin{definition} We say that the pair $\langle f(x), n\rangle $ satisfies the {\em reducible condition}, if $f(x)$ is a monic irreducible polynomial of degree $m$ and exponent $e$ and $$\nu_p(q-1) \ge \nu_p(n)+\nu_p(e)$$
for all $p$ prime divisor of $n$.
\end{definition}

The main result of this paper is the following

\begin{theorem}\label{redutivel}
Let $f(x)\in \F_q[x]$ be a monic  irreducible polynomial of degree $m$ and exponent $e$, and $p^t$ such that $\langle f(x), p^t\rangle $ satifies the reducible condition. Suposse that $k=\nu_p(e)$ and  $e=p^kr$. Then
\begin{enumerate}[(a)]
\item There exists  a unique element $c\in \F_q$ such that $f(x)$ divides $x^r-c$. In addition, $c=b^{p^t}$ for some $b\in \F_q$. Moreover, if $c=1$, then $b\in \mathcal U(p^t)$.
\item Let $A$ be any matrix with coefficients in $\F_q$ such that $f(x)$ is the characteristic polynomial of $A$,  $s$ be  the solution of $sr\equiv 1\pmod{p^t}$ with $0<s<p^t$, and $l=\frac {sr-1}{p^t}$. Then  $g(x)=\det(xI-b^{-s}A^l)\in \F_q[x]$ is a irreducible factor of $f(x^{p^t})$.
\item $a=b^{p^k}\in \U(p^t)$ and the polynomial $f(x^{p^t})$ splits into monic irreducible  factors over $\F_q[x]$ as
$$f(x^{p^t})=\prod_{j=0}^{p^t-1}[a^{-mj}g(a^jx)].$$
\end{enumerate}
\end{theorem}

\begin{proof}
Let $\alpha$ be the class of $x$ in the extension $\F_{q^ m}=\F_q[x]/f(x)$. 
Since $f(x)$ divides $x^e-1$,  it follows that $\alpha^e=1$ and $\ord(\alpha)=e=p^kr$. 

By hypotheses, $\langle f(x),p^t\rangle $ satisfies the reduction condition, or equivalently   $p^{k+t}$ divides $q-1$.  So, there exists an integer $L$ such that  $q=p^{k+t}L+1$.
From the fact that
$$(\alpha^r)^{(q-1)/p^t}=(\alpha^e)^{L}=1,$$ 
we obtain that  $\alpha^r\in \F_{q^ m}$ is a root of the polynomial $y^{(q-1)/p^t}-1=0$. But  the roots of that polynomial are the $p^t-$th perfect powers  in $\F_q$,  therefore, there exists $c=b^{p^t}$, with $b\in \F_q$, such that $\alpha^r=c\in \F_q$, or equivalently $x^r\equiv c\pmod{f(x)}$. It is easy to see that when $c=1$ we can take $b\in \U(p^t)$.
This proves item (a).

In order to prove  (b), since $\ord(\alpha)=e=p^kr$, by the previous item, we have that
$$
\ord(c)=\ord(\alpha^r)=p^k.$$
 Therefore, if $b\in \F_q$ is a $p^{t}-$th root of $c$ (or a primitive $p^t-$th root of unity, when $c=1$), then $b$ has order $p^{t+k}$ 
and so $b^{p^k}$ has order $p^t$, i.e., $a=b^{p^k}\in \U(p^t)$.
\par Taking $(-s)-$th powers in the equation $\alpha^r=c=b^{p^t}$, 
we obtain $\alpha^{-rs}=\alpha(\alpha^{-l})^{p^t}=b^{-p^ts}$, or equivalently  $\alpha=(\alpha^{l}b^{-s})^{p^t}$. If we set $\beta:= \alpha^lb^{-s}$ then $\beta^{p^t}=\alpha$.
 In particular, since  $f(\beta^{p^t})=f(\alpha)=0$, $\beta$ is a root of some irreducible factor $g(x)$ of $f(x^{p^t})$.
 By lemma \ref{frobenius} we have $$g(x)=\prod_{j=1}^{m}(x-\beta^{q^j})=\prod_{j=1}^{m}(x-(b^{-s}\alpha^l)^{q^j})=\prod_{j=1}^{m}(x-b^{-s}\alpha^{lq^j}).$$  

Now, if $A$ is a matrix such that $f(x)$ is the characteristic polynomial of $A$, since every roots of $f(x)$ are different, the matrix $A$ is diagonalizable and the eigenvalues are the roots of $f(x)$. 
It follows that the eigenvalues of  $b^{-s}A^l$  are of the form $b^{-s}\alpha^l$, where $\alpha$  runs through the roots  of $f(x)$, i.e. the eigenvalues of $b^{-s}A^l$ are the roots of the polynomial $g(x)$, and since $g(x)$ and $\det(x-b^{-s}A^l)$ are monic polynomials of  the same degree, we conclude $g(x)=\det(x-b^{-s}A^l)$.

Finally, since $a\in \U(p^t)$ and $g(x)$ is a monic irreducible factor of $f(x^{p^t})$, by lemma \ref{fatores} we have that
$$f(x^{p^t})=\prod_{j=0}^{p^t-1}[a^{-mj}g(a^jx)]$$ is the factorization of $f(x^{p^t})$ into irreducible factors over $\F_q$, as desired. \qed
\end{proof}

\begin{remark}
If $\langle f(x), n\rangle $ satisfies the reducible condition, where  $n=\prod_{i=1}^up_i^{\beta_i}$, then iterating the process  for each prime divisor, we obtain $n$  irreducible factors of $f(x^n)$. In addition, in Theorem \ref{redutivel} (b), we can take
$$A=\begin{pmatrix}0&1&0&\cdots&0\\ 0&0&1&\cdots&0\\
 \vdots&\vdots&\vdots&\ddots&\vdots\\ 
0&0&0&\cdots&1\\
-a_0&-a_1&-a_2&\cdots&-a_{m-1}
\end{pmatrix}$$
just the companion matrix of the polynomial   $f(x)=x^m+a_{m-1}x^{m-1}+\cdots+a_1x+a_0=\det(xI-A)$.
\end{remark}

\begin{corollary}\label{factorradical}  Let $f(x)\in \F_q$ be a monic irreducible polynomial of degree $m$ and exponent $e$. Let $n$ be a positive integer such that 
\begin{enumerate}
\item $\rad(n)$ divides $q-1$
\item $gcd(m, n)=1$
\item if $\nu_2(n)+\nu_2(e)\ge \nu_2(q-1)+2$ then $q\equiv 1 \pmod 4$
\end{enumerate}
Every factor of $f(x^n)$ is of the form $g(x^{n/\rho_{n,q,e}})$, where $\rho_{n,q,e}=\prod\limits_{p|n} p^{\min\{\nu_p(n), \nu_p(q-1)-\nu_p(e)\}}$  and $g(x)$ is any irreducible factor of $f(x^{\rho_{n,q,e}})$ 
obtained from the process described in Theorem \ref{redutivel}.
\end{corollary}
\begin{proof} Since $e$ divides $q^m-1$, for every $p$ prime divisor of $n$, we have that $gcd(p,m)=1$ and
$$\nu_p(e)\le \nu_p(q^m-1)=\nu_p(q-1)+\nu_p(m)=\nu_p(q-1),$$
hence it follows that $\rho_{n,q,e}$ is well defined  integer. In addition,  we can see $\rho_{n,q,e}$ as   the great divisor of $n$ such that  $\langle f(x), \rho_{n,q,e}\rangle $ satisfies the reducible condition.  
Now, suppose that  $g(x)\in \F_q[x]$ is an arbitrary  irreducible factor of $f(x^{\rho_{n,q,e}})$, then by Theorem \ref{redutivel}, we know that $g(x)$ has degree $m$ 
and exponent $e\cdot\rho_{n,q,e}$.  On the other hand,  for each $p$ prime divisor of  $\dfrac n{\rho_{n,q,e}}$, we have that  $\nu_p(\rho_{n,q,e})=\nu_p(q-1)-\nu_p(e)$
and then
$$\nu_p(q^m-1)=\nu_p(q-1)+\nu_p(m)=\nu_p(e\rho_{n,q,e}).$$
Thus, the previous relation implies that  $g(x^{n/\rho_{n,q,e}})$ satisfies the condition of Theorem \ref{irreducible} and therefore it is a monic irreducible  factor  of $f(x^n)$. \qed  
\end{proof} 

\begin{remark} Observe that if $q\equiv 3\pmod 4$, $f(x)\in\F_q[x]$ is a monic irreducible polynomial of odd degree $m$ and exponent $e$, and $n=2^t$ with $t+\nu_2(e)\ge \nu_2(q-1)+2$, 
then  we cannot applied the Corollary \ref{factorradical}  directly.  However,  since $f(x)$ is also an irreducible polynomial in $\F_{q^2}[x]$, 
  we can see $f(x)$ as an irreducible  polynomial with coefficient in $\F_{q^2}=\F_q(\sqrt{-1})$ and then,  using the Corollary,  we obtain irreducible factors of the form $h(x):=g(x^{n/\rho_{n,q^2,e}})\in \F_{q^2}[x]$. At this point, we have two cases to consider:  
\begin{itemize}
  \item If $h(x)\in \F_q[x]$, we have nothing else to do.
\item  If $h(x)\in \F_q^2[x]\setminus\F_q[x]$ then  $h(x)\cdot\overline{h(x)}\in \F_q[x]$  is an irreducible factor of $f(x^n)$ in  $\F_q[x]$, where $\overline{h(x)}$ means change $\sqrt{-1}$ by $-\sqrt{-1}$ in each coefficient of $h(x)$ where it appears. In fact, $\overline{h(x)}$  is equivalent to applying  the Frobenius map $\alpha\mapsto \alpha^q$ in each coefficient of $h(x)$.  
\end{itemize}
\end{remark}

\section{Numerical Examples}

\begin{example}
We consider $f(x)=x^2-11x+1\in \F_{59}[x]$. Since $disc(f(x))=11^2-4\equiv -1 \pmod{59}$ 
 and $-1$ is a quadratic non-residue, it follows that $f(x)$ has no roots, hence it is irreducible. 
Observe  that $\Phi_{12}(x)=(x^2-11x+1)(x^2+11x+1)$ and so $f(x)$ has exponent $e=12$. We are going to find the complete factorization of $f(x^{29^{d+1}})=x^{2\cdot 29^{d+1}}-11x^{29^{d+1}}+1$ for all $d\ge 1$.

Firstly, we are going to consider the case $d=0$. Notice that gcd$(12, 29)=1$ and $29|(59-1)$.
Using  the notation of Theorem \ref{redutivel}, we have $r=12$ and $12s\equiv 1\pmod {19}$. Directly, we can see that $s=17$ is a solution of this congruence   and then we can define $l=\frac{rs-1}{29}=7$. 
Now, we just have to find some element in $\mathcal U(29)\subset\F_{59}$. 
Since $59-1=2\times 29$, it follows that any quadratic residue $b\ne 1$, must be a primitive $29-$th root of unity. Using the quadratic reciprocity law, it is  easy verify  that  $5$ is a quadratic residue. So, using the notation of our Theorem, we have $k=0$ and $a=b=5$.\\
Finally, notice that 
$\displaystyle    A=
  {\begin{pmatrix}
   0 & 1 \\       -1 & 11 \\      \end{pmatrix} } 
$
is the companion matrix of $f(x)$ and, in particular, the characteristic polynomial of $A$ is $f(x)$. By the  Theorem, we know that 
$$g(x)=\det (xI-b^{-s}A^l)=\det (xI-3^{-17}A^{7})$$ is a factor of $f(x^{29})$. 
Using a direct computation we find 
$$
   A^{7}=
  {\begin{pmatrix}
   0 & -1 \\       1 & -11 \\     \end{pmatrix} }=-A\ \text{ and }\ 
5^{-17}\equiv 41\pmod{59}.$$ 
So $g(x)=\det(xI+18A)=  \left| {\begin{array}{cc} x & -18 \\       18 & x-21 \\      \end{array} } \right|=x^2-21x+18^2=x^2-21x+29$. \\

Moreover, every monic irreducible factors of $f(x^{29})$ are of the form $$g_j(x)=5^{-2j}g(5^{j}x)=5^{-2j}(25^jx^2-21\times 5^j x+29)=x^2-(21\times 5^{-j})x+29\times 5^{-2j}$$
where $j=0,\dots, 28$. In other words
$$x^{58}-11x^{29}+1=\prod_{i=0}^{28}(x^2-(21\times 12^{j})x+29\times 26^{j}).$$

Now, observe that each factor $g_j(x)$ has degree $2$ and exponent  $12\times 29$, then by  Theorem \ref{irreducible}, the polynomials $g_j(x^{29^d})$ are irreducible. Therefore, 
$$f(x^{29^{d+1}})=\prod_{i=0}^{28}(x^{2\cdot 29^d}-(21\times 12^{j})x^{29^d}+29\times 26^{j})$$
is the decomposition into irreducible factors.
\end{example}

\begin{example}
We consider $f(x)=x^4+x^3+x^2+x+1=\Phi_5(x)\in \F_{q}[x]$, where $q=2^{2^{3}}+1=257$ is the fourth Fermat prime. 
Since $q\equiv2\pmod 5$ and $2$ is primitive element of $\Z_5$, it follows that $f(x)$ is irreducible and has exponent $5$. 

We are going to split the polynomial  $f(x^{2^{8}})=x^{1024}+x^{768}+x^{512}+x^{256}+1$ into irreducible factors. 
Observe that  gcd$(5, 2^{8})=1$ and $2^{8}|q-1$, then $\langle f(x),2^8\rangle $ satisfies  the reducible  condition and then we can apply  Theorem \ref{redutivel}.
Using  the notation of the Theorem, we have $r=5$ and $5s\equiv 1\pmod {2^{8}}$. So,   $s=205$ and $l=\frac{rs-1}{2^{8}}=4$. 
Now,  we proceed to find an element an  element in $\mathcal U(2^{8})\subset \F_q$. Since $q-1=2^{8}$, it follows that any  quadratic non-residue $b$, must be a primitive $2^{8}-$th root of unity.
 By the Quadratic Reciprocity Law, it is easy to verify that $b=3$ is a quadratic non-residue and  again, using the notation of the Theorem, we have $k=0$ and $a=b=3$.
Finally, we notice that 
 \[
   A=\begin{pmatrix}
   0 & 1 & 0 & 0 \\ 0& 0 & 1 & 0\\ 0 & 0 & 0 & 1\\ -1 & -1& -1 & -1 \\      \end{pmatrix} 
\] 
is the companion matrix of $f(x)$ and, then $f(x)$ is  the characteristic polynomial of $A$. Therefore, we have  that $g(x)=\det (xI-b^{-s}A^l)=\det (xI-3^{-s}A^{4})$ is a factor of $f(x^{2^{8}})$. By a direct calculation, we obtain
 $$B=
   A^{4}=  \begin{pmatrix}
   -1 & -1 & -1 & -1 \\ 1 & 0 & 0 & 0\\ 0 & 1 & 0 & 0\\ 0 & 0 & 1 & 0 \\      \end{pmatrix}  
\ \text{ and }\ 3^{-205}=3^{51}\equiv 54 \pmod{257}.
$$
 So $$g(x)=\det(xI-54 B)=   \left| {\begin{array}{cccc}
  x+54 & 54 & 54 & 54 \\ -54 & x & 0 & 0\\ 0 & -54 & x & 0\\ 0 & 0 & -54 & x \\      \end{array} } \right|=x^4+203 x^3+89 x^2+77 x+211 $$
\end{example}

The previous example is a particular case when $\Phi_{p^t}(x)\in \F_q[x]$ is an irreducible polynomial. This case will be consider in the next section, where we are going to explicit every irreducible factor without use of the companion matrix of the polynomial. 

\section{Explicit  factors of  $x^{2^np^t}-1$}
In this section, we give an explicit factorization of polynomials $x^{2^np^t}-1\in \F_q[x]$, where $p$ is an odd prime, $q$ is primitive element modulo  $p^2$,  $n\le \nu_2(q-1)$ and $t\ge 1$ is any fixed integer.  We note that a more general result was proved by Tuxanidy and  Wang \cite{TuWa}, where they obtain the explicit factorization of the cyclotomic polynomial $\Phi_{2^nr}(x)\in \F_q$, where $r\ge 3$ is  odd with $gcd(r,q)=1$ and $n$ is a positive integer, under the assumption that the explicit factorization of $\Phi_r(x)\in\F_q[x]$ is known.

Since $x^d-1=\prod_{m|d}\Phi_m(x)$, we put our focus on cyclotomic polynomials. First, let us recall two simple (but strong) results about cyclotomic polynomials 

\begin{lemma}[\cite{LiNi}, Exercise 2.57]\label{lem5.1}
Let $\Phi_d(x)$ the $d-$th cyclotomic polynomial over $\F_q[x]$. Then
\begin{enumerate}[(a)]
\item $\Phi_{2d}(x)=\Phi_d(-x)$ for $d>3$ and $d$ odd.
\item $\Phi_{mp^t}(x)=\Phi_{mp}(x^{p^{t-1}})$ if $p$ is a prime and $m, t$ are arbitrary positive integers.
\end{enumerate}
\end{lemma}

\begin{lemma}[\cite{LiNi}, Theorem 2.47 (ii)]\label{ciclotomico_irr}
Let $d$ be a positive integer such that gcd$(q, d)=1$, then $\Phi_d(x)$  splits into $\varphi(d)/s$ distinct monic irreducible polynomials in $\F_q[x]$ of the same degree $s$, where $s=\ord_dq$.
\end{lemma}

It is a well known result in number theory that if $q$ is a primitive  element $\Z_{p^2}$ then $q$ is also a primitive element in  $\Z_{p^k}$ for any $k\ge 1$. 
Therefore, by our hypothesis, $q$ is a primitive element  modulo $ {p^j}$  for each $j\ge 1$ and then, by Lemma \ref{ciclotomico_irr}, $\Phi_{p^j}(x)$ is an irreducible polinomial for all $1\le j\le k$. By the recursive relation of ciclotomic  polynomials, we know that
$$\Phi_{p^j}(x)=\prod_{i<p^j, (i, p)=1}(x-\zeta^{i})=\frac{x^{p^j}-1}{x^{p^{j-1}}-1}= x^{p^{j-1}(p-1)}+ x^{p^{j-1}(p-2)}+\cdots+ x^{p^{j-1}}+1,$$
and  by Lemma \ref{lem5.1}, we obtain
$$\Phi_{2p^j}(x)=\Phi_{p^j}(-x)=\prod_{i<p^j, (i, p)=1}(-x-\zeta^{i})= x^{p^{j-1}(p-1)}- x^{p^{j-1}(p-2)}+\cdots- x^{p^{j-1}}+1.$$
It follows that, if $n=0$ or $n=1$, then
$$x^{2^np^t}-1=\prod\limits_{0\le i\le 1 \atop 0\le j\le t}\Phi_{2^ip^j}(x)$$ 
is the decomposition of the polynomial $x^{2^np^t}-1$ into   irreducible factors.
  
Therefore, we only need to find the factor of $\Phi_{2^ip^j}(x)$, where   $2\le i\le \nu_2(q-1)$ and $j\ge 1$. By Lemma \ref{lem5.1} we know that $\Phi_{2^ip^j}(x)=\Phi_{2p^j}(x^{2^{i-1}})=\Phi_{p^j}(-x^{2^{i-1}})$.
 So, the problem is equivalent to factorize $f(x^{2^{i-1}})$ where $f(x)=\Phi_{2p^j}(x)$. 
Observe that, the exponent of $f$ is  $e=2p^j$ and $\nu_2(e)+\nu_2(2^{(i-1)})= i\le \nu_2(q-1)$,  therefore $\langle  f(x), 2^{i-1}\rangle$ satisfies the reducible condition. 
Thus, by Lemma \ref{fatores}, $f(x^{2^{i-1}})$ splits  into $2^{i-1}$ irreducible factors and, by  the same Lemma, it is enough to find one irreducible factor in order to determine  all  other.


Now, suppose that $\zeta$ is primitive $p^j$-th root of the unity in some field $\mathbb K\supset \F_q$, then
\begin{align}\Phi_{2^ip^j}(x)=\Phi_{p^j}(-x^{2^{i-1}})&=(-1)^{\varphi(2p^j)}\prod_{s<p^j, (s, p)=1}(x^{2^{i-1}}+\zeta^{s})\label{p_ciclo}\\
&= \prod_{s<p^j, (s, P)=1}(x^{2^{i-1}}+\zeta^{s})\nonumber
\end{align}
In addition,  both  sets $\{s; s<p^j, (s,p)=1\}$ and $\{2^{i-1}s; s<p^j, (s,p)=1\}$ generate the same set of  invertible residues module ${p^j}$  and $\U(2^{i-1})\subset \F_q$ is not empty, for all $i\le n$.  We point out that, an explicit element of $\U(2^{i-1})$ can be found  choosing a random element  $\theta\in \F_q^*$ that is not a square in $\F_q$, or equivalently, $\theta$ such that  $\theta^{(q-1)/2}\ne 1$, and we define $b_i:=\theta^{(q-1)/2^{i-1}}\in\U(2^{i-1})$.
This element $\theta$ can be found  after  $d$  random choices with probability $1-\frac 1{2^d}$.  Thus, from (\ref{p_ciclo})  follows that
$$\Phi_{2p^j}(x^{2^{i-1}})=\prod_{s<p^j, (s, p)=1}(x^{2^{i-1}}-b_i^{2^{i-1}}\zeta^{s2^{i-1}})$$
Now, notice that 
$$g(x)=\prod_{s<p^j, (i, p)=1}(x-b_i\zeta^{s})=\sum_{m=0}^{p-1}b_i^{(p-1-m)p^{j-1}}x^{mp^{j-1}}\in \F_q[x]$$
 is a monic polynomial and it divides $\Phi_{2p^j}(x^{2^{i-1}})=f(x^{2^{i-1}})$. Since $b_i\in \U(2^{i-1})$ and  $\deg f(x)=\varphi(2p^j)=\varphi(p^j)$, by Lemma \ref{fatores}  the decomposition of the polynomial  as
$$\Phi_{2^ip^j}(x)=f(x^{2^{i-1}})=\prod_{l=0}^{2^{i-1}-1}[b_i^{-l\varphi(p^j)}g(xb_i^l)],$$
where the explicit forms of each factor $b_i^{-l\varphi(p^j)}g(xb_i^l)$ is
$$b_i^{-l\varphi(p^j)}g(xb_i^l)=b_i^{-l\varphi(p^j)}\sum_{m=0}^{p-1}b_i^{(p-1-m)p^{j-1}}(xb_i^l)^{mp^{j-1}}=\sum_{m=0}^{p-1}b_i^{(1-l)(p-1-m)p^{j-1}}x^{mp^{j-1}}.$$

\section{Efficiency of the process and iterations}
With this process, find the irreducible factors  of $f(x^{p^t})$ requires essentially five step: taking $p^t-$th roots and $n-$th powers of elements in a finite field, calculating $x^r\pmod {f(x)}$ and taking a power of the companion matrix  and the characteristic polynomial of a  $m\times m$  matrix over a finite field. 

Now we discuss the algorithm complexity of each operations in term of number of multiplications in $\F_q$, without any consideration about the characteristic of the field. We point out that, we can improve the complexity calculation in term of the characteristic $q_0$ and $u$, where $q=q_0^u$.

\subsection{Taking powers in $\F_q$  and calculate  $x^r\pmod {f(x)}$}

If $a\in\F_q$, taking squares successively is a well-know fast process for finding $a^n$ in essentially  $2\log_2(n)$ products of elements in $\F_q$.  On the other hand,   product of two polynomial and reduction modulo $f(x)$ can be done with $O(m\log^2 m\log\log m)$  products in $\F_q$ using fast Euclidean algorithm and  Cantor-Kaltofen Algorithm (see \cite{CaKa}). It follows that for calculate $x^r\pmod{f(x)}$ when $r> m$,  we need $O(m\log \frac rm\log^2 m\log\log m)$ product  in $\F_q$.

\subsection{Taking roots in $\F_q$}
Using  Tonelli-Shanks probabilistic Algorithm  (see \cite{Ton, Sha}) we can take square root over $\F_q$  with $(4+2u)\log q+\frac{\nu_2(q-1)^2}4$ products in $\F_q$ with probability $1-\frac 1{2^u}$ (see \cite{Tor}).
Iterating the process we find $2^t-$th roots of elements in a finite field  with $(4t+2u)\log q+t\frac{\nu_2(q-1)^2}4$ products in $\F_q$ with probability $1-\frac 1{2^u}$.

More generally, Kaltofen and Shoup \cite{KaSh}  show a generic probabilistic algorithm for taking $n-$th roots in a finite field.  This algorithms use with high probability, $O(n^{1.815}\log q)$ products in $\F_q$. Other algorithm can be found in \cite{DoSc}.

\subsection{Powers of a square matrix}

In our theorem, we have to find $A^l$ where $A$ is the companion matrix of $f(x)$, a $m\times m$ matrix. There is a lot of efficient algorithms for taking powers of a matrix as the well known Coppersmith-Winograd Algorithm, whose complexity is $O(m^{2.4}\log l)$.
 But, in the case of companion matrices, we may have a faster process. 
In \cite{Roy} is given an algorithm that generates new rows to the matrix $A$ and each row is obtained after $m$ multiplications and $m-1$ additions. 
Moreover, it is shown that we just have to calculate the first $l-1$ new arrows to find $A^l$. Thus we just need $(2m-1)(l-1)$ operations to find $A^l$.

\subsection{Characteristic Polynomial}

In our theorem, we have to calculate exactly one characteristic polynomial of a $m\times m$ matrix. The classical methods of  Krylov's and  Danilevsky  have complexity $O(m^3)$. Specifically,  $\frac 32m^2(m+1)$ and $(m-1)(m^2+m-1)$ multiplications and divisions respectively.

In conclusion, if $\langle f(x), n\rangle$ satisfies the reducible condition, where $m=\deg(f(x))$ and  using that $r<q^m$  and $l<n$, we obtain that the complexity of the algorithm to find every irreducible factors of $f(x^n) \in \F_q[x]$  has complexity bounded by
$$O(m^2\log^2 m\log\log m \log q,\,  n^{1.815}\log q,\,  mn,\,  m^3).$$

\end{document}